\documentclass[12pt]{article}%
\usepackage{amsmath}
\usepackage{amsfonts}
\usepackage{amssymb}
\usepackage{graphicx}%
\providecommand{\U}[1]{\protect\rule{.1in}{.1in}}
\newtheorem{theorem}{Theorem}

\newtheorem{corollary}[theorem]{Corollary}

\newtheorem{definition}[theorem]{Definition}

\newtheorem{proposition}[theorem]{Proposition}
\newtheorem{remark}[theorem]{Remark}

\newenvironment{proof}[1][Proof]{\noindent\textbf{#1.} }{\ \rule{0.5em}{0.5em}}
\begin{document}

\title{Steady waves in flows over periodic bottoms}
\author{Walter Craig \thanks{This work was submitted posthumously on behalf of the
first author.}
\and Carlos Garc\'{\i}a-Azpeitia \thanks{Depto. Matem\'{a}ticas y Mec\'{a}nica
IIMAS, Universidad Nacional Aut\'{o}noma de M\'{e}xico, Apdo. Postal 20-726,
01000 Ciudad de M\'{e}xico, M\'{e}xico. cgazpe@mym.iimas.unam.mx} }
\maketitle

\begin{abstract}
We study the formation of steady waves in two-dimensional fluids under a
current with mean velocity $c$ flowing over a periodic bottom. Using a
formulation based on the Dirichlet-Neumann operator, we establish the unique
continuation of a steady solution from the trivial solution when a flat
bottom is perturbed, except for a sequence of velocities $c_{k}$. The main
contribution is the proof that at least two steady solutions exist close to
a non-degenerate $S^{1}$-orbit of non-constant steady waves when a flat
bottom is perturbed. Consequently, we obtain persistence of at least two
steady waves close to a non-degenerate $S^{1}$-orbit of Stokes waves
bifurcating from the velocities $c_{k}$.
\end{abstract}

\section{Introduction}

Stokes' analysis of periodic water waves in a region of infinite depth
heralded much interest in the field of fluid dynamics. In the early
twentieth century, Nekrasov and Levi-Civita first rigorously proved the
existence of the Stokes waves (two-dimensional $2\pi $-periodic gravity
waves on water of infinite depth). Stokes waves are steady solutions when
they are viewed in a reference frame moving with speed $c$. This result was
extended to the case of finite flat bottoms by Struik. These waves appear as
a bifurcation from the trivial solution for the velocities 
\begin{equation*}
c_{k}:=\left( g\frac{1}{\left\vert k\right\vert }\tanh (h\left\vert
k\right\vert )\right) ^{1/2},\qquad k\in \mathbb{N},
\end{equation*}%
where $g$ is the acceleration of gravity and $h$ is the deepness of the
bottom. Stokes conjectured that the primary branch from $c_{1}$ is limited
by an extreme wave with a singularity at their crest forming a $2\pi /3$%
-angle. The first rigorous global bifurcation of traveling waves (Stokes
waves) was presented in \cite{Krasovskii2}, but Stokes conjecture was proved
only recently in \cite{AFT} and \cite{P}.

For varying bottoms, the Euler equation is not invariant by translations
anymore, i.e., the problem cannot be reduced to a steady wave in a moving
reference frame. On the other hand, for $2\pi $-periodic bottoms $-h+b(x)$,
steady waves exist in a stream current of mean velocity $c$ in a fixed
reference frame. The formation of steady waves in two-dimensional flows for
near-flat bottoms (with small $b$) has been studied previously in \cite%
{Gerber}, \cite{Krasovskii1}, \cite{Moiseev} and \cite{Shinbrot}. The first
rigorous result in \cite{Gerber} proves the existence of one steady solution
for bottoms with two extrema per period, except for the sequence of
degenerate values $c_{k}$. Those results were extended in \cite{Krasovskii1}
to any near-flat bottom and to general bottoms in the case that $c$ is large
enough. Under similar hypothesis, according to \cite{Moiseev}, the top of
the fluid follows the bottom when $c>c_{\ast }$, but it inverts when $%
c<c_{\ast }$. Later on, the existence of steady solutions for a
three-dimensional fluid was analyzed in \cite{Shinbrot}.

The purpose of this work is to investigate further the existence of
two-dimensional steady waves over a periodic bottom. We formulate the Euler
equation as a Hamiltonian system similarly to \cite{Za}. The work \cite{La}
contains a short exposition of the different formulations of the Euler
equation. We formulate the Hamiltonian using the Dirichlet-Neumann operator
and a mean stream current with velocity $c$ analogously to the formulations
in \cite{CrGu}, \cite{CrNi1} and \cite{CrSu1}. Using this approach, in
Theorem \ref{Thm1} we recover the result in \cite{Krasovskii1} regarding the
unique continuation of the trivial solution for small perturbations of
bottom $b$, except for the sequence of velocities $c_{k}$.

The Hamiltonian for water waves in flat bottoms $b=0$ is $S^{1}$-invariant,
where the group 
\begin{equation*}
S^{1}:=\mathbb{R}/2\pi \mathbb{Z}
\end{equation*}%
acts by translations in the periodic domain $[0,2\pi ]$. The $S^{1}$%
-invariance of the Hamiltonian implies that its Hessian at the trivial
solution with velocity $c_{k}$ has a two-dimensional kernel. The paper \cite%
{Ig} presents an analysis of the set of solutions near the bifurcation point 
$c=c_{k}$ by classifying the patterns of bifurcation according to the shape
of the bottom, but only in the space of even surfaces and for even bottoms,
which simplifies the problem because the kernel becomes one dimensional
under these constraints. The analysis of the bifurcation diagram for general
surfaces and small bottoms near $c=c_{k}$ is difficult because the kernel is
two-dimensional, but the Hamiltonian is not $S^{1}$-invariant anymore.
Indeed, even to determine the splitting of linear eigenvalues of the Hessian
near the degenerate velocities $c_{k}$ is a challenging task; for example,
this phenomenon is analyzed in \cite{CrSu2}. On the other hand, our
Hamiltonian formulation allows us to prove the persistence of solutions far
from the degenerate velocities. Specifically,

\begin{figure}[t]
\resizebox{14cm}{!}{\hspace{-4cm} \includegraphics{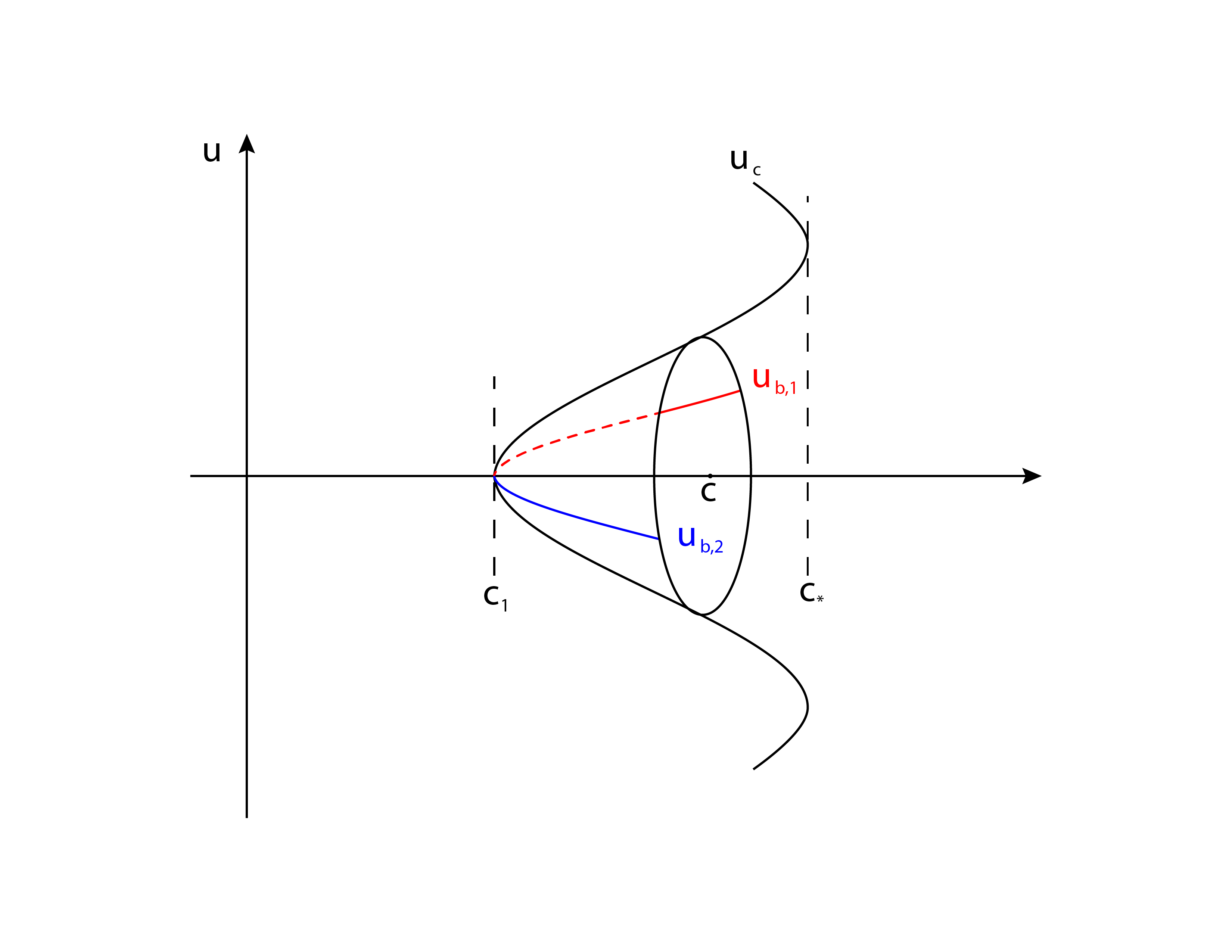} \hspace{-8cm}
\includegraphics{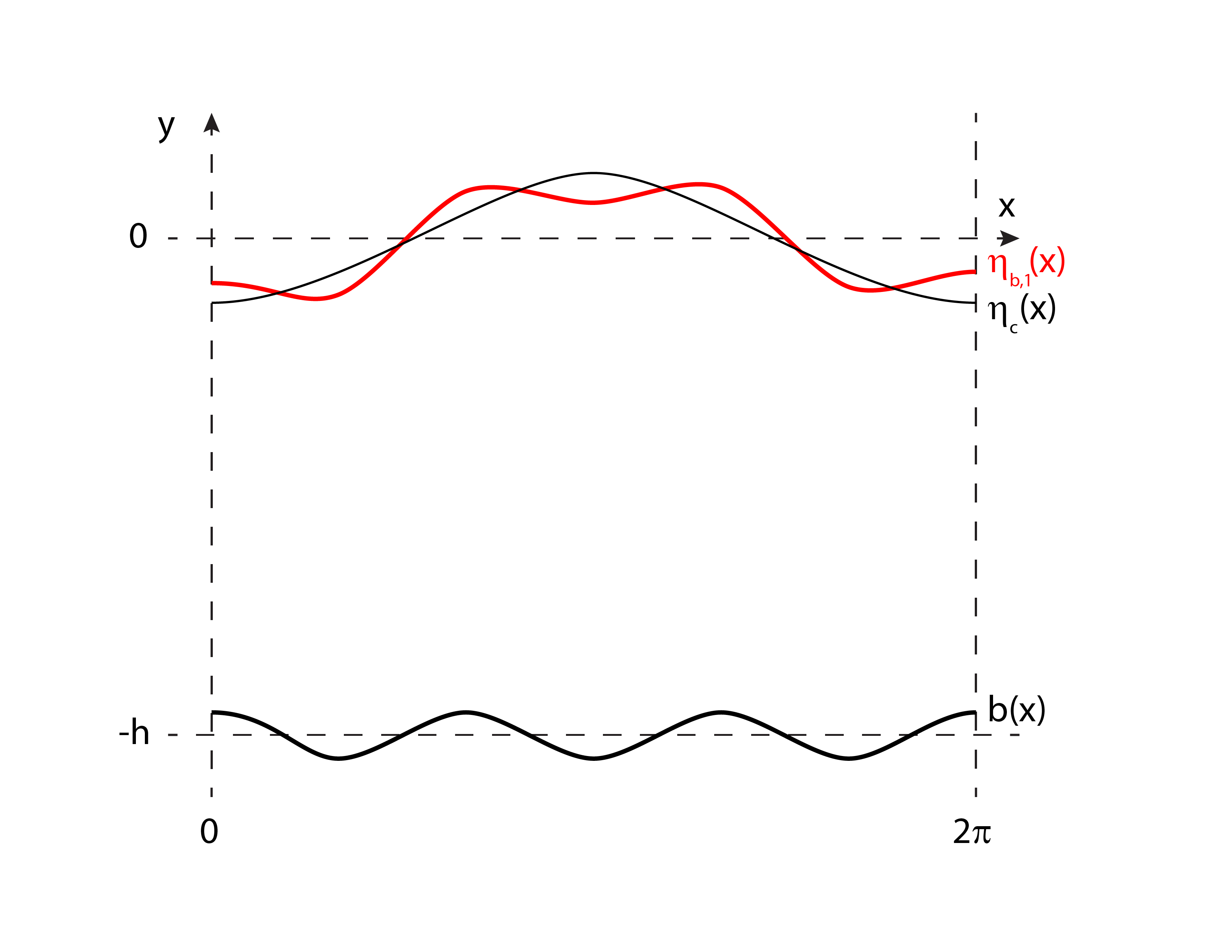} } 
\caption{Left: Bifurcation diagram of the family of $S^{1}$-orbits of
solutions $u_{c}$ (Stokes wave). At least two solutions $u_{b,1}$ (red) and $%
u_{b,2}$ (blue) persist when $b$ is small, except for degenerate values such
as $c_{1}$ and $c_{\ast}$. Right: Illustration of the surface $\protect\eta%
_{b,1}(x)$ (red)$\ $near the surface $\protect\eta_{c}(x)$ (Stokes wave).}
\label{fig2}
\end{figure}

\textbf{Main Result.} \emph{In Theorem \ref{Thm3} we prove that at least two
steady waves exist close to the non-degenerate }$S^{1}$\emph{-orbit
(Definition \ref{nondeg}) of a non-constant steady wave when a flat bottom
is perturbed. Consequently, this theorem implies the persistence of at least
two steady waves close to the primary branch of Stokes waves which are
non-degenerate. This phenomenon is illustrated in Figure \ref{fig2}. }

We will briefly discuss the non-degeneracy property of an $S^{1}$-orbit of
the primary branch of Stokes waves. The appearance of a degenerate $S^{1}$%
-orbit in a branch of Stokes waves leads generically to the existence of
bifurcation (Figure \ref{fig2}). For a fluid with infinite depth, the
article \cite{BuDaTo} proves that the Morse index of solutions along the
primary branch of Stokes waves bifurcating from $c_{1}$ diverges as the
branch approaches the extremal wave (conjectured by Stokes). These analytic
methods imply that the branch of Stokes waves, parameterized by $\lambda $,
has a countable number of critical values $\left( \lambda _{j}\right) _{j\in 
\mathbb{N}}$ containing turning points (fold bifurcations)\ or harmonic
bifurcations. The numerical evidence is that only turning points occur.
Therefore, the results in \cite{BuDaTo} imply that the primary branch of
Stokes waves is non-degenerate (in the subspace of even periodic functions)
except for set of critical values $\left( \lambda _{j}\right) _{j\in \mathbb{%
N}}$.

To the best of our knowledge, the bifurcation from the branch of Stokes
waves has not been established rigorously for a fluid with a flat bottom ($%
b=0$). The numerical computations in \cite{CrNi2} indicate that only a
countable set of turning points occur in the primary branch of Stokes waves
arising from $c_{1}$. Given this numerical evidence, we conjecture that the $%
S^{1}$-orbits of the primary branch of Stokes waves, parameterized by $%
\lambda $, is non-degenerate except for a countable set of critical values $%
\left( \lambda _{j}\right) _{j\in \mathbb{N}}$. \emph{If this conjecture is
true, our theorem proves the persistence of at least two steady waves close
to the primary branch of Stokes waves, except for the hypothetical set of
critical values }$\left( \lambda _{j}\right) _{j\in \mathbb{N}}$\emph{. }

A possible approach to prove this conjecture is to extend the methods in 
\cite{BuDaTo} to the Babenko formulation for a fluid with a flat bottom in 
\cite{KuDi}. But this challenging task is beyond the scope of our
presentation. Our main contribution is to present a novel mathematical
framework to prove the persistence of steady waves different to the ones
studied previously from a flat surface. These steady waves can be found
close to the Stokes waves or even its secondary bifurcations. These steady
waves have been observed and described in the formation of dunes (out of
phase waves) and in antidunes (in phase waves).

Theorem \emph{\ref{Thm3}} is proved by applying a Lyapunov-Schmidt reduction
in a Sobolev space of $2\pi $-periodic function to solve the normal
components to a non-degenerate $S^{1}$-orbit. The result is obtained from
the fact that the reduced Hamiltonian (defined in the domain $S^{1}$) has at
least two critical points. The Hamiltonian formulation presented in this
work can be used to study other problems of interest such as the existence
of steady waves in more dimensions. However, for torus domains $\mathbb{T}%
^{n}$ it is necessary to consider surface tension to avoid the small divisor
problem, see \cite{CrNi1} for references. In such a case, one can apply
Lusternik-Schnirelmann category to prove the persistence of at least $n+1$
solutions from a non-degenerate $\mathbb{T}^{n}$-orbit of steady waves
(Stokes waves). Lusternik-Schnirelmann category has been used to prove the
persistence of solutions in finite-dimensional Hamiltonian systems in \cite%
{FoMo}.

The paper is organized as follows. In Section 2, we define the Hamiltonian
using the Dirichlet-Neumann operator. In Section 3, we prove the
continuation of the trivial solution for small $b$. In Section 4, we prove
the persistence of at least two steady waves near a non-degenerate $S^{1}$%
-orbit of steady waves for small $b$.

\section{Formulation for two-dimensional steady water waves}

We study the problem of steady waves in the periodic domain $S^{1}=\mathbb{R}%
/2\pi\mathbb{Z}$. In this domain, we define the Sobolev space of $2\pi$%
-periodic functions, 
\begin{equation*}
H^{s}\left( S^{1};\mathbb{R}\right) =\left\{ u(x)=\sum_{k\in\mathbb{Z}%
}u_{k}e^{ixk}\in L^{2}:\left\vert u\right\vert _{s}^{2}=\sum_{k\in\mathbb{Z}%
}\left( 1+k^{2}\right) ^{s}\left\vert u_{k}\right\vert ^{2}<\infty\right\}
,\qquad s\geq0\text{.}
\end{equation*}
The free surface of the fluid is represented by the curve $y=\eta(x)$, and
the bottom of the fluid by $y=-h+b(x)$, where $b$ is a $2\pi$-periodic
variation from the mean deep $-h$.

In the domain%
\begin{equation}
D_{b,\eta }=\left\{ \left( x,y\right) \in S^{1}\times \mathbb{R}%
:-h+b(x)<y<\eta (x)\right\} \text{,}
\end{equation}%
the Hamiltonian for the time-dependent Euler equation is given by $H=K+E$,
where the kinetic and potential energy are 
\begin{equation*}
K=\int_{S^{1}}\left( \int_{-h+b}^{\eta }\frac{1}{2}\left\vert \nabla \phi
\right\vert ^{2}dy\right) dx,\qquad E=\int_{S^{1}}\frac{1}{2}g\eta ^{2}~dx%
\text{.}
\end{equation*}%
According to \cite{Wi}, the requirement that the functional $H$ is
stationary with respect to independent variations $\delta \eta $ and $\delta
\phi $ gives the set of equations%
\begin{eqnarray*}
\Delta \phi &=&0\text{ in }D_{b,\eta }\text{,} \\
\frac{\partial \phi }{\partial n} &=&n\cdot \nabla \phi =0\text{ in }%
\partial D_{b,\eta }\text{,}
\end{eqnarray*}%
and the kinematic boundary condition and Bernoulli condition at $y=\eta (x)$%
, 
\begin{eqnarray*}
\partial _{t}\eta &=&\partial _{y}\phi -\partial _{x}\eta \cdot \partial
_{x}\phi , \\
\partial _{t}\phi &=&-\frac{1}{2}\left\vert \nabla \phi \right\vert
^{2}-g\eta ,
\end{eqnarray*}%
see also \cite{Cr} for details.

The Zakharov formulation assigns a Hamiltonian structure to the dynamics of
waves. Zakharov discovered in \cite{Za} that a subtle aspect is the choice
of the canonical variables $\xi (x,t)=\phi (x,\eta (x,t),t)$ and $\eta (x,t)$
for the phase space of the time-dependent Euler equation satisfying the
kinematic boundary condition $\partial _{t}\eta =\delta _{\xi }H$ and the
Bernoulli condition $\partial _{t}\xi =-\delta _{\eta }H$. In particular,
the steady waves are critical points of the Hamiltonian $H(\eta ,\xi )$ in
the domain $D_{b,\eta }$.

Furthermore, if we look for critical points of $H(\eta ,\xi )$ restricted by
the constraint of zero excess of mass 
\begin{equation*}
m(\eta ):=\int_{S^{1}}\eta (x)=0\in \mathbb{R}\text{,}
\end{equation*}%
then these critical points satisfy instead the equations $\delta _{\xi }H=0$
and $\delta _{\eta }H+\lambda \delta _{\eta }m(\eta )=0$, where $\lambda $
is the Lagrange multiplier. Since $\delta _{\eta }m(\eta )=1$, the second
equation is up to a constant the Bernoulli condition 
\begin{equation*}
\delta _{\eta }H+\lambda =\frac{1}{2}\left\vert \nabla \phi \right\vert
^{2}+g\eta +\lambda =0.
\end{equation*}%
Hereafter, we study the problem of steady waves as critical points of the
Hamiltonian $H(\eta ,\xi )$ restricted to the space of functions with zero
excess of mass $\eta \in H_{0}^{s+1}$, where%
\begin{equation*}
H_{0}^{s}=\left\{ \eta \in H^{s}:\int_{S^{1}}\eta (x)~dx=0\right\} ,\qquad
s\geq 0.
\end{equation*}

\subsection{Dirichlet-Neumann operator}

To obtain an expression for the Hamiltonian $H(\eta ,\xi )$ we require the
Dirichlet-Neumann operator. Hereafter the symbols $N_{\eta }\cdot \nabla $
and $N_{b}\cdot \nabla $ represent the external normal derivatives (not
normalized) at the top and bottom of the domain $D_{b,\eta }$,%
\begin{equation*}
N_{\eta }\cdot \nabla =\partial _{y}-\partial _{x}\eta \partial _{x},\qquad
N_{b}\cdot \nabla =-\partial _{y}+\partial _{x}b\partial _{x}~.
\end{equation*}

\begin{definition}
\label{DN}The Dirichlet-Neumann operator is defined by%
\begin{equation*}
G(\eta ;b)\xi =\left. N_{\eta }\cdot \nabla \Phi \right\vert _{y=\eta (x)},
\end{equation*}%
where $\Phi $ is the unique harmonic function in $D_{b,\eta }$ that
satisfies the boundary condition $\xi (x)=\Phi (x,\eta (x))$ at the top $%
y=\eta (x)$ and $N_{b}\cdot \nabla \Phi =0$ at the bottom $y=-h+b(x)$.
\end{definition}

By the continuous embedding $H^{s+1}\hookrightarrow C^{1}(S^{1};\mathbb{R})$
with $s>1/2$, there is a constant $\gamma $ such that $\left\vert \cdot
\right\vert _{C^{1}}\leq \gamma \left\vert \cdot \right\vert _{H^{s+1}}$.
Therefore $\left\vert b\right\vert _{C^{0}}\leq \gamma \varepsilon $ if $%
\left\vert b\right\vert _{s+1}\leq \varepsilon $. The condition $\left\vert
b\right\vert _{s+1}\leq \varepsilon $ and $\eta \geq -h+2\gamma \varepsilon $
imply that the bottom $y=-h+b(x)$ and the top $y=\eta (x)$ remain at a $%
\gamma \varepsilon $-distance,%
\begin{equation}
\left\vert \eta -\left( -h+b\right) \right\vert _{C^{0}}\geq \left\vert
h+\eta \right\vert _{C^{0}}-\left\vert b\right\vert _{C^{0}}\geq \gamma
\varepsilon \text{.}  \label{dist}
\end{equation}

\begin{proposition}
By Theorem A.11 in \cite{La}, the Dirichlet-Neumann operator 
\begin{equation*}
G(\eta ;b):H^{s+1}\rightarrow H^{s},\qquad s>1/2,
\end{equation*}%
is analytic as a function of $(\eta ;b)\in H_{0}^{s+1}\times H^{s+1}$ if $%
\left\vert b\right\vert _{s+1}\leq \varepsilon $ with $s>1/2$ and $\eta
(x)>-h+2\gamma \varepsilon $ for $x\in S^{1}$.
\end{proposition}

\subsection{The current of mean velocity $c$}

We define the cylinder-like domain 
\begin{equation}
D_{b}=\left\{ \left( x,y\right) \in S^{1}\times \mathbb{R}:-h+b(x)<y\right\}
.
\end{equation}%
Notation $b\lesssim a$ means that there is a positive constant $C$ such that 
$b\leq Ca$ for all $a>0$ sufficiently small.

In section 2.4 we prove the following estimate for the harmonic function $%
\Phi _{b}$.

\begin{theorem}
\label{ThmA}There is an $\varepsilon >0$ such that if $\left\vert
b\right\vert _{s+1}\leq \varepsilon $ with $s>1/2$, then there is a unique
harmonic function $\Phi _{b}\left( x,y\right) :D_{b}\rightarrow \mathbb{R}$
with boundary conditions 
\begin{equation}
\lim_{y\rightarrow \infty }\nabla \Phi _{b}(x,y)=0,\qquad \lim_{y\rightarrow
\infty }\Phi _{b}(x,y)=0,  \label{Harm1}
\end{equation}%
and 
\begin{equation}
N_{b}\cdot \nabla \left( \Phi _{b}+x\right) =0~,\qquad y=-h+b(x).
\label{Harm2}
\end{equation}%
Furthermore, the function $\Phi _{b}\left( x,y\right) $ satisfies the
estimate 
\begin{equation}
\left\vert \Phi _{b}\right\vert _{C^{k}(\bar{D})}\lesssim \left\vert
b\right\vert _{s+1}~,\qquad k\in \mathbb{N}\text{,}
\end{equation}%
where $\bar{D}\subset D_{2\gamma \varepsilon }$ is a compact set.
\end{theorem}

If $b\in H^{s+1}$ with $s>1/2$, we define $\nabla \Psi _{c}=c(x+\Phi _{b})$
as the current of mean velocity $c$. The function $\Psi _{c}$ is the unique
harmonic function in $D_{b}$ with boundary conditions%
\begin{equation*}
\lim_{y\rightarrow \infty }\nabla \Psi _{c}(x,y)=(c,0),\qquad
\lim_{y\rightarrow \infty }\Psi _{c}(x,y)=cx,
\end{equation*}%
and zero Neumann boundary condition at the bottom 
\begin{equation*}
N_{b}\cdot \nabla \Psi _{c}=0~,\qquad y=-h+b.
\end{equation*}

\begin{remark}
Actually, the harmonic function $\Psi _{1}=x+\Phi _{b}$ generates the De
Rham $1$-cohomology group $H^{1}(D_{b})=\mathbb{R}$ in the cylindrical
domain $D_{b}$. Indeed, $H^{1}(D_{b})$ is generated by the form $\Omega
=-\partial _{y}\Psi _{1}dx+\partial _{x}\Psi _{1}dy$ that is closed, $%
d\Omega =\Delta \Psi _{1}dx\wedge dy=0$, but not exact. The Hodge-Helmholtz
decomposition \cite{Marsden} implies that any vector field $u$ defined in
the domain $D_{b}$ with the boundary conditions $\lim_{y\rightarrow \infty
}u=(c,0)$ and $N_{b}\cdot \nabla u=0$ at $y=-h+b(x)$ is decomposed into
three components: a divergence-free (incompressible), a rotation-free
(irrotational), and a harmonic (translational) component. The component $%
\nabla \Psi _{1}$ is harmonic in the sense that it is divergence-free $%
\nabla \cdot \nabla \Psi _{1}=\Delta \Psi _{1}=0$ and curl-free $\nabla
^{\perp }\cdot \nabla \Psi _{1}=0$, where $\nabla ^{\perp }=(-\partial
_{y},\partial _{x})$. For more dimensions, the domain $D_{b}=\left\{ \left(
x,y\right) \in \mathbb{T}^{n}\times \mathbb{R}:-h+b(x)<y\right\} $ has $n$
translational fields $\nabla \Psi _{j}$ corresponding to the De Rham $1$%
-cohomology $H^{1}(D_{b})=\mathbb{R}^{n}$. Our analysis can be extended to
torus domains of general dimension $\mathbb{T}^{n}$ by considering these $n$
fields $\nabla \Psi _{j}$ that represent the currents in the different
directions of the domain $D_{b}$.
\end{remark}

\subsection{The gradient of the Hamiltonian}

The kinetic energy for $\phi =\Phi +\Psi _{c}$ (where $\Phi $ is a
perturbation of the current generated by $\Psi _{c}$) is given by 
\begin{equation*}
K=\int_{S^{1}}\int_{-h+b}^{\eta }\frac{1}{2}\left\vert \nabla \phi
\right\vert ^{2}dy~dx=\int_{S^{1}}\int_{-h+b}^{\eta }\frac{1}{2}\left\vert
\nabla \Phi +\nabla \Psi _{c}\right\vert ^{2}dy~dx.
\end{equation*}%
We are ready to establish the formulation of the kinetic energy $K$ in terms
of the variables $(\eta ,\xi )\in H_{0}^{s+1}\times H^{s+1}$ and the
parameters $b\in H^{s+1}$ and $c\in \mathbb{R}$.

\begin{proposition}
If $\eta $ has zero-average, then the kinetic energy in terms of the
variables $\eta (x)$ and $\xi (x)=\Phi (x,\eta (x))$, where $\Phi $ is the
harmonic function in Definition \ref{DN}, is given by%
\begin{equation*}
K=\int_{S^{1}}\left( \frac{1}{2}\xi G(\eta ;b)\xi +c\xi \left( N_{\eta
}\cdot \nabla \Phi _{b}-\partial _{x}\eta \right) +\frac{c^{2}}{2}\eta +%
\frac{c^{2}}{2}\Phi _{b}\left( N_{\eta }\cdot \nabla \Phi _{b}-2\partial
_{x}\eta \right) \right) dx+C,
\end{equation*}%
where $\Phi _{b}$ and $N_{\eta }\cdot \nabla \Phi _{b}$ are functions
evaluated at $y=\eta (x)$ and $C$ is a constant depending on $b\in H^{s+1}$
and $c\in \mathbb{R}$, but independent of $(\eta ,\xi )$.
\end{proposition}

\begin{proof}
The kinetic energy is%
\begin{equation*}
K=\int_{D_{b,\eta }}\frac{1}{2}\nabla \Phi \cdot \nabla \left( \Phi +2\Psi
_{c}\right) +\int_{D_{b,\eta }}\frac{1}{2}\left\vert \nabla \Psi
_{c}\right\vert ^{2}~.
\end{equation*}%
Since $\Phi $ and $\Psi _{c}$ are harmonic functions, by the Divergence
Theorem we have%
\begin{equation*}
K=\int_{\partial D_{b,\eta }}\frac{1}{2}\Phi \nabla \left( \Phi +2\Psi
_{c}\right) \cdot N~dS+\int_{D_{b,\eta }}\frac{1}{2}\left\vert \nabla \Psi
_{c}\right\vert ^{2}\text{,}
\end{equation*}%
where $N$ is the normal of $\partial D_{b,\eta }$. Since $\nabla \left( \Phi
+2\Psi _{c}\right) $ is $2\pi $-periodic, then 
\begin{equation*}
K=\frac{1}{2}\int_{S^{1}}\left. \Phi N_{\eta }\cdot \nabla \left( \Phi
+2\Psi _{c}\right) \right\vert _{y=\eta }+\left. \Phi N_{b}\cdot \nabla
\left( \Phi +2\Psi _{c}\right) \right\vert _{y=-h+b}dx+\int_{D_{b,\eta }}%
\frac{1}{2}\left\vert \nabla \Psi _{c}\right\vert ^{2}\text{.}
\end{equation*}

Given that $\Phi $ and $\Psi _{c}$ satisfy the Neumann boundary conditions
at the bottom, while $N_{\eta }\cdot \nabla \Phi =G(\eta ;b)\xi $ and $%
N_{\eta }\cdot \nabla \Psi _{c}=c(N_{\eta }\cdot \nabla \Phi _{b}-\partial
_{x}\eta )$ at the top $y=\eta (x)$, the kinetic energy becomes%
\begin{equation*}
K=\int_{S^{1}}\frac{1}{2}\xi G(\eta ;b)\xi +c\xi \left( \left( N_{\eta
}\cdot \nabla \Phi _{b}\right) _{y=\eta }-\partial _{x}\eta \right)
~dx+\int_{D_{b,\eta }}\frac{1}{2}\left\vert \nabla \Psi _{c}\right\vert ^{2},
\end{equation*}%
where%
\begin{equation*}
\int_{D_{b,\eta }}\frac{1}{2}\left\vert \nabla \Psi _{c}\right\vert
^{2}=\int_{D_{b,\eta }}\frac{1}{2}c^{2}\left\vert \nabla (\Phi
_{b}+x)\right\vert ^{2}=\frac{c^{2}}{2}\int_{D_{b,\eta }}\left( \nabla \Phi
_{b}\cdot \nabla \left( \Phi _{b}+2x\right) +1\right) .
\end{equation*}%
Since $\eta $ has zero average, then 
\begin{equation*}
\int_{D_{b,\eta }}1=\int_{S^{1}}(\eta (x)+h-b(x))~dx=2\pi \left(
h-\int_{S^{1}}b\right) .
\end{equation*}%
Thus, by the Divergence Theorem, and the fact that $\nabla \left( \Phi
_{b}+x\right) _{y=-h+b}=0$, we have%
\begin{equation*}
\int_{D_{b,\eta }}\frac{1}{2}\left\vert \nabla \Psi _{c}\right\vert ^{2}=%
\frac{c^{2}}{2}\int_{S^{1}}\left( \eta (x)+\left( \Phi _{b}~N_{\eta }\cdot
\nabla \left( \Phi _{b}+2x\right) \right) _{y=\eta }~+\left( \Phi
_{b}~\partial _{x}b\right) _{y=-h+b}\right) dx+\pi hc^{2}.
\end{equation*}%
The result follows with the constant 
\begin{equation*}
C:=\pi c^{2}\left( h-\int_{S^{1}}b\right) +\frac{c^{2}}{2}%
\int_{S^{1}}\partial _{x}b(x)\Phi _{b}(x,-h+b(x))dx\text{.}
\end{equation*}
\end{proof}

\begin{corollary}
\label{Co1}By dropping the constant $C$, the Hamiltonian in $(\eta ,\xi )$%
-coordinates is given by%
\begin{equation}
H(\eta ,\xi ;b,c)=K(\eta ,\xi )+E(\eta )=\hat{H}(\eta ,\xi )+\tilde{H}(\eta
,\xi )\text{.}
\end{equation}%
Here, 
\begin{align*}
\hat{H}(\eta ,\xi ;b,c)& =\int_{S^{1}}\left( \frac{1}{2}\xi G(\eta ;b)\xi
-c\xi \partial _{x}\eta +\frac{1}{2}g\eta ^{2}\right) ~dx, \\
\tilde{H}(\eta ,\xi ;b,c)& =\int_{S^{1}}\left( \frac{c^{2}}{2}\eta +c\xi
N_{\eta }\cdot \nabla \Phi _{b}+\frac{c^{2}}{2}\Phi _{b}N_{\eta }\cdot
\nabla \Phi _{b}-c^{2}\Phi _{b}\partial _{x}\eta \right) ~dx,
\end{align*}%
where $\Phi _{b}$ and $N_{\eta }\cdot \nabla \Phi _{b}$ are functions
evaluated at $y=\eta (x)$.
\end{corollary}

For a flat bottom with $b=0$ we have that $\tilde{H}(\eta ,\xi ;0,c)=0$,
then the Hamiltonian becomes%
\begin{equation*}
H(\eta ,\xi ;0,c)=\hat{H}(\eta ,\xi ;0,c)=\int_{S^{1}}\left( \frac{1}{2}\xi
G(\eta ;0)\xi -c\xi \partial _{x}\eta +\frac{1}{2}g\eta ^{2}\right) dx.
\end{equation*}%
Therefore, for a flat bottom the Hamiltonian $H(\eta ,\xi ;0,c)$ is the same
Hamiltonian that is used in \cite{CrNi1} to study the existence of traveling
waves arising from the trivial solution at speed $c_{k}$ (Stokes waves).

\begin{remark}
It is important to mention that the Euler equation is equivalent to the
Hamiltonian system $\dot{\eta}=\partial _{\xi }H(u;b,c)$ and $\dot{\xi}%
=-\partial _{\eta }H(\eta ,\xi ;b,c)$ for the flat bottom $b=0$, see \cite%
{CrNi1} for details. This is not true for $b\neq 0$ because $\eta $ and $\xi 
$ are not conjugated variables anymore. The conjugated variables can be
obtained using the Legendre transformation in the Lagrangian $L=K-E$.
Nevertheless, the solutions of $\nabla _{(\eta ,\xi )}H(\eta ,\xi ;b,c)=0$
are steady solutions of the Euler equation.
\end{remark}

In next proposition we define the space where the gradient of the
Hamiltonian is well defined.

\begin{proposition}
The gradient map%
\begin{equation*}
\nabla _{(\eta ,\xi )}H(\eta ,\xi ;b,c):H_{0}^{s+1}\times H_{0}^{s+1}\times
H^{s+1}\times \mathbb{R}\rightarrow H_{0}^{s}\times H_{0}^{s}
\end{equation*}%
is well defined if $\left\vert b\right\vert _{s+1}\leq \varepsilon $ with $%
s>1/2$ and $\eta >-h+2\gamma \varepsilon $.
\end{proposition}

\begin{proof}
By Corollary \ref{Co1} and the computations in reference \cite{CrSu2}, we
have that 
\begin{equation*}
\nabla _{(\eta ,\xi )}\hat{H}(\eta ,\xi ;b,c)=\left( 
\begin{array}{c}
c\partial _{x}\xi +g\eta +\frac{1}{2}\left\vert \partial _{x}\xi \right\vert
^{2}-\frac{\left( G(\eta ;b)\xi +\partial _{x}\eta \partial _{x}\xi \right)
^{2}}{2(1+\left( \partial _{x}\eta \right) ^{2})^{2}} \\ 
-c\partial _{x}\eta +G(\eta ;b)\xi%
\end{array}%
\right) .
\end{equation*}%
We can compute the partial derivative 
\begin{equation*}
\partial _{\xi }\tilde{H}(\eta ,\xi ;b,c)=c\left( \partial _{y}\Phi
_{b}(x,\eta )-\partial _{x}\eta \partial _{x}\Phi _{b}(x,\eta )\right) .
\end{equation*}%
Moreover, the partial derivative $\partial _{\eta }\tilde{H}(\eta ,\xi ;b,c)$
is a sum of products of $\xi $, $\eta $, $\partial _{x}\eta $, $\partial
_{x}\xi $ and $\Phi _{b}$ and its derivatives evaluated at $y=\eta (x)$.

By Section 2.1. the Dirichlet--Neumann operator $G(\eta ;b)$ is bounded from 
$H^{s+1}$ to $H^{s}$ because $\left\vert b\right\vert _{s+1}\leq \varepsilon 
$ with $s>1/2$ and $\eta >-h+2\gamma \varepsilon $. The analytic function $%
\Phi _{b}(x,y)$ satisfies the estimate of Theorem \ref{ThmA} at $y=\eta (x)$
because $\eta (x)>-h+2\gamma \varepsilon $ for any $x\in S^{1}$. Therefore,
by the Banach Algebra property of $H^{s}$ for $s>1/2$ and the analyticity of 
$\Phi _{b}$ at $y=\eta (x)$, the gradient operator 
\begin{equation*}
\nabla _{(\eta ,\xi )}H(\eta ,\xi ;b,c):H_{0}^{s+1}\times H_{0}^{s+1}\times
H^{s+1}\times \mathbb{R}\rightarrow H_{0}^{s}\times H^{s},
\end{equation*}%
is well defined. It only remains to prove that $\partial _{\xi }H\in
H_{0}^{s}$. This fact follows from the Divergence Theorem, 
\begin{align*}
\int_{S^{1}}\partial _{\xi }H~dx& =\int_{S^{1}}\partial _{\xi
}K~dx=\int_{S^{1}}G(\eta ;b)\xi +c\left( N_{\eta }\cdot \nabla \Phi
_{b}-\partial _{x}\eta \right) _{y=\eta }~dx \\
& =\int_{S^{1}}\left( \nabla \Phi +\nabla \Psi _{c}\right) _{y=\eta }\cdot
N_{\eta }~dx=\int_{D_{b,\eta }}\Delta \left( \Phi +\Psi _{c}\right) =0\text{.%
}
\end{align*}
\end{proof}

\subsection{Estimates of the current of mean velocity $c$}

\begin{proof}[Proof of Theorem \protect\ref{ThmA}]
We obtain the estimates for the function $\Phi _{b}$ using the methods
developed in \cite{CrGu} by means of the Green function $G$ in $D_{0}$. The
fundamental solution of Laplace equation in the domain $(x,y)\in S^{1}\times 
\mathbb{R}$ is given by the Green function%
\begin{equation*}
G_{per}(x,y)=\frac{1}{4\pi }\ln \left( \sin ^{2}x+\sinh ^{2}y\right) \text{.}
\end{equation*}%
Let $y^{\ast }=-y-2h$ be the image of $y$ reflected at the line $y=-h$. By
the method of images, the Green function in the domain $D_{0}$ satisfying
the Neumann boundary condition at the bottom $y=-h$ is given by 
\begin{align*}
G(x-x^{\prime },y,y^{\prime })& =\frac{1}{4\pi }\ln \left( \sin
^{2}(x-x^{\prime })+\sinh ^{2}(y-y^{\prime })\right) \\
& +\frac{1}{4\pi }\ln \left( \sin ^{2}(x-x^{\prime })+\sinh ^{2}(y+y^{\prime
}+2h)\right) .
\end{align*}%
The Green function $G$ is defined up to a constant. We normalized $G$ by the
condition that 
\begin{equation}
\lim_{y\rightarrow +\infty }G(x-x^{\prime },y,y^{\prime })=0\text{.}
\label{B1}
\end{equation}%
Notice that by construction, the Green function $G(x-x^{\prime },y,y^{\prime
})$ is analytic except for singularities at $x^{\prime }=x$ and $y^{\prime
}=y$ or $y^{\prime }=2h-y$ (the reflection of $y$ at the line $y=-h$).

We use the Green function $G$ to obtain the estimates of the harmonic
function $\Phi _{b}$ satisfying the boundary conditions (\ref{Harm1}) and (%
\ref{Harm2}). Explicitly, the boundary condition (\ref{Harm2}) is 
\begin{equation*}
N_{b}\cdot \nabla \Phi _{b}(x,-h+b(x))=-\partial _{x}b(x).
\end{equation*}%
Using (\ref{B1}) and the fact that $\Phi _{b}$ and $G$ are $2\pi $-periodic,
we have by Green's identity that 
\begin{align}
\Phi _{b}(x,y)& ={\scriptstyle\int_{D_{b}}\Phi _{b}(x^{\prime },y^{\prime
})\Delta _{(x^{\prime },y^{\prime })}G(x-x^{\prime },y,y^{\prime
})-G(x-x^{\prime },y,y^{\prime })\Delta _{(x^{\prime },y^{\prime })}\Phi
_{b}(x^{\prime },y^{\prime })\label{green}} \\
& ={\scriptstyle\int_{S^{1}}\left[ \Phi _{b}(x^{\prime },y^{\prime
})N_{b}\cdot \nabla _{(x^{\prime },y^{\prime })}G(x-x^{\prime },y,y^{\prime
})-G(x-x^{\prime },y,y^{\prime })N_{b}\cdot \nabla \Phi _{b}(x^{\prime
},y^{\prime })\right] _{y^{\prime }=-h+b(x^{\prime })}~}dx^{\prime }\text{.}
\notag
\end{align}%
This Green identity reads 
\begin{equation*}
\Phi _{b}=-\mathcal{A}[b]+\mathcal{B}[b]\Phi _{b},
\end{equation*}%
where 
\begin{equation}
\mathcal{A}[b](x,y)=\int_{S^{1}}G(x-x^{\prime },y,-h+b(x^{\prime }))\partial
_{x^{\prime }}b(x^{\prime })~dx^{\prime }\text{,}  \label{comp1}
\end{equation}%
and 
\begin{equation*}
\mathcal{B}[b]\Phi _{b}(x,y)=\int_{S^{1}}N_{b}\cdot \nabla _{(x^{\prime
},y^{\prime })}G(x-x^{\prime },y,-h+b(x^{\prime }))\Phi _{b}(x^{\prime
},-h+b(x^{\prime }))~dx^{\prime }.
\end{equation*}

We have explicitly that the Green function is 
\begin{align*}
G(x-x^{\prime },y,-h+b(x^{\prime }))& =\frac{1}{4\pi }\ln \left( \sin
^{2}(x-x^{\prime })+\sinh ^{2}(y+h-b(x^{\prime }))\right) \\
& +\frac{1}{4\pi }\ln \left( \sin ^{2}(x-x^{\prime })+\sinh
^{2}(y+h+b(x^{\prime })\right) ,
\end{align*}%
and the normal derivative of the Green function is 
\begin{align}
N_{b}\cdot \nabla _{(x^{\prime },y^{\prime })}G(x-x^{\prime
},y,-h+b(x^{\prime }))& ={\scriptstyle-\frac{1}{2\pi }\frac{b^{\prime
}(x^{\prime })\sin 2\left( x-x^{\prime }\right) }{\cosh 2\left( b(x^{\prime
})-h-y\right) -\cos 2\left( x-x^{\prime }\right) }-\frac{1}{2\pi }\frac{%
b^{\prime }(x^{\prime })\sin 2\left( x-x^{\prime }\right) }{\cosh 2\left(
b(x^{\prime })+h+y\right) -\cos 2\left( x-x^{\prime }\right) }}  \notag \\
& {\scriptstyle-\frac{1}{2\pi }\frac{\sinh 2\left( b(x^{\prime })-h-y\right) 
}{\cosh \left( b(x^{\prime })-h-y\right) -\cos \left( x-x^{\prime }\right) }-%
\frac{1}{2\pi }\frac{\sinh \left( b(x^{\prime })+h+y\right) }{\cosh \left(
b(x^{\prime })+h+y\right) -\cos \left( x-x^{\prime }\right) }}\text{.} 
\notag
\end{align}%
Since $b\in C^{1}$, the function 
\begin{equation*}
N_{b}\cdot \nabla _{(x^{\prime },y^{\prime })}G(x-x^{\prime
},y,-h+b(x^{\prime }))_{y=-h+b(x)}
\end{equation*}%
has a singularity of order $\frac{1}{x^{\prime }-x}$. Notice that $\mathcal{B%
}[b]\Phi _{b}^{\ast }$ is defined by an integral that depends only on the
values of $\Phi _{b}^{\ast }:\partial D_{b}\rightarrow \mathbb{R}$ at the
bottom $\partial D_{b}=\left\{ \left( x,-h+b(x)\right) :x\in S^{1}\right\} $%
. Therefore, by applying Korn-Lichtenstein theorem to principal value of the
integral $\mathcal{B}[b]\Phi _{b}^{\ast }$, we obtain that 
\begin{equation*}
\mathcal{B}[b]\Phi _{b}^{\ast }:Z\rightarrow Z,
\end{equation*}%
where $Z=C^{0,\alpha }(\partial D_{b})$ is the space of Holder continuous
functions in $\partial D_{b}$.

Furthermore, we used the Green function $G(x-x^{\prime },y,y^{\prime })$
with Neumann boundary condition at $y=-h$ to have%
\begin{equation*}
N_{b}\cdot \nabla _{(x^{\prime },y^{\prime })}G(x-x^{\prime
},y,-h+b(x^{\prime }))|_{b=0}=0.
\end{equation*}%
This fact implies that the operator norm of $\mathcal{B}[b]$ is bounded as
follows, 
\begin{equation*}
\left\Vert \mathcal{B}[b]\right\Vert _{Z\rightarrow Z}\lesssim \left\vert
b\right\vert _{C^{1}}\lesssim \left\vert b\right\vert _{s+1}\text{.}
\end{equation*}%
On the other hand, the function $G(x-x^{\prime },y,-h+b(x^{\prime
}))_{y=-h+b(x)}$ has a logarithmic singularity at $x^{\prime }=x$, which
implies that the function in definition (\ref{comp1}) is integrable in the
classical sense and $\left\vert \mathcal{A}[b]\right\vert _{Z}\lesssim
\left\vert b\right\vert _{C^{1}}$. Applying Neumann series to the operator $%
I-\mathcal{B}[b]$ for $\left\vert b\right\vert _{s+1}\leq \varepsilon $ with 
$\varepsilon <<1$, we obtain that $I-\mathcal{B}[b]$ is invertible and $\Phi
_{b}^{\ast }=-\left( I-\mathcal{B}[b]\right) ^{-1}\mathcal{A}[b]\in Z$
satisfies the estimate 
\begin{equation*}
\left\vert \Phi _{b}^{\ast }\right\vert _{Z}\lesssim \left\vert b\right\vert
_{C^{1}}\lesssim \left\vert b\right\vert _{s+1}~.
\end{equation*}

The functions $G(x-x^{\prime },y,-h+b(x^{\prime }))$ and $N_{b}\cdot \nabla
_{(x^{\prime },y^{\prime })}G(x-x^{\prime },y,-h+b(x^{\prime }))$ are
analytic for any $(x,y)\in D_{b}$. This implies that $\mathcal{A}[b](x,y)$
and $\mathcal{B}[b]\Phi _{b}^{\ast }(x,y)$ are analytic functions for $%
(x,y)\in D_{b}$. By construction the function $\Phi _{b}(x,y):=\mathcal{B}%
[b]\Phi _{b}^{\ast }(x,y)+\mathcal{A}[b](x,y)$ extends $\Phi _{b}^{\ast
}(x,y)$, is harmonic in $D_{b}$, and satisfies the boundary conditions (\ref%
{Harm1}) and (\ref{Harm2}). By (\ref{dist}) the bottom $y=-h+b(x)$ remains
at a $\gamma \varepsilon $-distance from the domain $\bar{D}\subset
D_{2\gamma \varepsilon }$ if $\left\vert b\right\vert _{s+1}\leq \varepsilon 
$. Thus, the fact that the functions $G(x-x^{\prime },y,-h+b(x^{\prime }))$
and $N_{b}\cdot \nabla _{(x^{\prime },y^{\prime })}G(x-x^{\prime
},y,-h+b(x^{\prime }))$ are analytic in the compact set $(x,y)\in \bar{D}$
implies the estimates $\left\vert \mathcal{A}[b]\right\vert _{C^{k}(\bar{D}%
)}\lesssim \left\vert b\right\vert _{s+1}$ and $\left\vert \mathcal{B}%
[b]\Phi _{b}\right\vert _{C^{k}(\bar{D})}\leq C\left\vert b\right\vert
_{s+1}\left\vert \Phi _{b}^{\ast }\right\vert _{Z}$. Since $\left\vert \Phi
_{b}^{\ast }\right\vert _{Z}\lesssim \left\vert b\right\vert _{s+1}$, from
the identity $\Phi _{b}=-\mathcal{A}[b]+\mathcal{B}[b]\Phi _{b}^{\ast }$ we
obtain that 
\begin{equation*}
\left\vert \Phi _{b}\right\vert _{C^{k}(\bar{D})}\lesssim \left\vert
b\right\vert _{s+1}
\end{equation*}
\end{proof}

\begin{remark}
We prove that if $\left\vert b\right\vert _{s+1}\leq \varepsilon $ with $%
s>1/2$, then there is a constant $C$ such that $\left\vert \Phi
_{b}\right\vert _{C^{k}(\bar{D})}\leq C\left\vert b\right\vert _{s+1}$. If
there is some $x_{0}\in S^{1}$ such that $b(x_{0})=2\gamma \varepsilon $,
then the Green function $G(x-x_{0},y,-h+b(x_{0}))$ explodes as $%
(x,y)\rightarrow (x_{0},-h+2\gamma \varepsilon )\in \partial D_{2\gamma
\varepsilon }$. This implies that the constant $C\rightarrow \infty $ as $%
\left\vert b\right\vert _{C^{0}}\rightarrow 2\gamma \varepsilon $. We used
the condition $\left\vert b\right\vert _{s+1}\leq \varepsilon $ to guarantee
that $b$ remains at a $\gamma \varepsilon $-distance from $\bar{D}\subset
D_{2\gamma \varepsilon }$ and the constant $C$ obtained from $%
G(x-x_{0},y,-h+b(x_{0}))$ does not blow up.
\end{remark}

There is an analytic expression of the Green function in terms of $b\in
H^{s+1}$ given by%
\begin{equation*}
G(x-x^{\prime },y,-h+b(x^{\prime }))=\sum_{m=0}^{\infty
}A_{m}[b](x,y,x^{\prime }),
\end{equation*}%
where $A_{m}[b]$ is a homogeneous function of degree $m$ in $b$. Thus $%
\mathcal{A}[b]$ is given by 
\begin{equation*}
\mathcal{A}[b](x,y)=\sum_{m=0}^{\infty }\int_{S^{1}}A_{m}[b](x,y,x^{\prime
})\partial _{x^{\prime }}b(x^{\prime })~dx^{\prime }\text{.}
\end{equation*}%
For example, we have that 
\begin{equation*}
A_{0}[b](x,y,x^{\prime })=\frac{1}{2\pi }\ln \left( \sin ^{2}(x-x^{\prime
})+\sinh ^{2}(y+h)\right) ,
\end{equation*}%
and $A_{1}[b](x,y,x^{\prime })=0$. Similarly, the operator $\mathcal{B}[b]$
has an analytic expression in powers of $b$ given by 
\begin{equation*}
\mathcal{B}[b]\Phi _{b}=\sum_{m=0}^{\infty
}\int_{S^{1}}B_{m}[b](x,y,x^{\prime })\Phi _{b}(x^{\prime },-h+b(x^{\prime
}))dx^{\prime },
\end{equation*}%
where $B_{m}[b]$ are homogeneous function of degree $m$ in $b$ such that 
\begin{equation*}
N_{b}\cdot \nabla _{(x^{\prime },y^{\prime })}G(x-x^{\prime
},y,-h+b(x^{\prime }))=\sum_{m=0}^{\infty }B_{m}[b](x,y,x^{\prime }).
\end{equation*}%
The Neumann boundary condition of the Green function $G(x-x^{\prime
},y,y^{\prime })$ implies that $B_{0}[b]=0$. In our proof we do not require
that $\mathcal{A}[b]$ and $\mathcal{B}[b]$ are analytic for $b\in H^{s+1}$,
we only used that $B_{0}[b]=0$.

\section{Continuation from the trivial solution}

To simplify the notation, we denote 
\begin{equation}
u=(\eta ,\xi )\in X:=H_{0}^{s+1}\times H_{0}^{s+1},\qquad \nabla
_{u}H(u;b,c)\in Y:=H_{0}^{s}\times H_{0}^{s}.  \label{X}
\end{equation}%
The gradient $\nabla H(u;0,c)$ has the trivial solution $u=0$. However, the
Hamiltonian $H=\hat{H}+\tilde{H}$ depends on $b$ through the term $\tilde{H}%
(u;b,c)$ that represents the interaction with the bottom $b$. Since $\tilde{H%
}$ contains terms of order $u$, then $u=0$ is not a trivial solution for $%
b\neq 0$ anymore. In this section we prove that the trivial solution can be
continued for small $b$ in the case that the Hessian $D^{2}H(0;0,c)$ is
invertible. This happens except for certain values of $c$, denoted by $c_{k}$%
. We continue the solution for $c\neq c_{k}$ by an application of the
Implicit Function Theorem.

From \cite{CrSu1} and \cite{CrNi1}, the expansion of $\hat{H}(u;0,c)$ in
power series is 
\begin{equation*}
\hat{H}(u;0,c)=\int_{S^{1}}\left( \frac{1}{2}\xi G(0;0)\xi -c\xi \partial
_{x}\eta +\frac{1}{2}g\eta ^{2}+h.o.t.\right) dx,
\end{equation*}%
where the linear operator $G(0;0)\xi $ in Fourier components $\xi
(x)=\sum_{k\in \mathbb{Z}}\xi _{k}e^{ikx}$ is given by 
\begin{equation*}
G(0;0)\xi =\sum_{k\in \mathbb{Z}}\left\vert k\right\vert \tanh (h\left\vert
k\right\vert )\xi _{k}e^{ikx}.
\end{equation*}

\begin{theorem}
\label{Thm1}For each regular velocity $c$ such that $c\neq c_{k}$, where%
\begin{equation*}
c_{k}^{2}:=g\frac{1}{\left\vert k\right\vert }\tanh (h\left\vert
k\right\vert ),\qquad k\in \mathbb{Z}\backslash \{0\},
\end{equation*}%
there is a unique steady wave $u_{b}\in X$ that is the continuation of the
trivial solution $u=0$ for $b$ in a small neighborhood of $0\in H^{s+1}$.
\end{theorem}

\begin{proof}
The gradient map $\nabla H\ $is well defined if $\left\vert b\right\vert
_{s+1}\leq \varepsilon $ with $s>1/2$ and $\eta >-h+2\gamma \varepsilon $.
We have that $\nabla H(0;0,c)=0$ for all $c\in \mathbb{R}$. Since $\Phi
_{0}=0$ for $b=0$, then $\tilde{H}(0;0,c)=0$ and 
\begin{equation*}
D^{2}H(0;0,c)=D^{2}\hat{H}(0;0,c)=L(c)\text{,}
\end{equation*}%
where 
\begin{equation*}
L(c):=\left( 
\begin{array}{cc}
g & c\partial _{x} \\ 
-c\partial _{x} & G(0;0)%
\end{array}%
\right) \text{.}
\end{equation*}%
Since $X$ and $Y$ do not have the $0$th Fourier component, the operator $L$
has the representation $L(c)u=\sum_{k\in \mathbb{Z}\backslash
\{0\}}A_{k}u_{k}e^{ikx}$, where $u(x)=\sum_{k\in \mathbb{Z}\backslash
\{0\}}u_{k}e^{ikx}$ and 
\begin{equation*}
A_{k}=\left( 
\begin{array}{cc}
g & ick \\ 
-ick & \left\vert k\right\vert \tanh (h\left\vert k\right\vert )%
\end{array}%
\right) .
\end{equation*}

For $k\in\mathbb{Z}\backslash\{0\}$ the matrix $A_{k}$ has determinant 
\begin{equation*}
\det A_{k}(c)=g\left\vert k\right\vert \tanh(h\left\vert k\right\vert
)-(ck)^{2}\text{.}
\end{equation*}
Since $\det A_{k}(c)=0$ at $c=c_{k}$, then the operator $L(c)=D^{2}\hat {H}%
(0;0,c)$ is invertible if $c\neq c_{k}$ (see Theorem 3.1 in \cite{CrNi1} for
details). Therefore, for a fixed $c\neq c_{k}$, the implicit function
theorem gives the unique solution $u_{b}(x)$ of the equation $\nabla
H(u;b,c)=0$ in a neighborhood of $(u;b)=(0;0)$.
\end{proof}

This theorem recovers the continuation result in \cite{Krasovskii1} for a
near-flat bottom.

\section{Continuation from non-degenerate steady waves}

For a flat bottom $b=0$, the Hamiltonian $H(u;0,c)$ is $G$-invariant under
the action 
\begin{equation*}
\varphi \cdot u(x)=u(x+\varphi ),\qquad \varphi \in G:=\mathbb{R}/2\pi 
\mathbb{Z}.
\end{equation*}%
Since the kernel of the Hessian $D^{2}H(0;0,c_{k})$ is a non-trivial $G$%
-representation, it has dimension two. The continuation of the trivial
solutions for small $b$ and velocities $c_{k}$ is a challenging task due to
the existence of the two-dimensional kernel. On the other hand, our
Hamiltonian formulation allows us to prove the persistence of solutions for
non-trivial $G$-orbits of solutions. This result can be applied to the
branch of Stokes waves that arises from $(u;c)=(0;c_{k})$. The local
bifurcation of Stokes waves is proved in \cite{CrNi1} using the Hamiltonian $%
H(u;0,c)$. The global property of the bifurcation can be obtained using this
setting and $G$-equivariant degree theory developed in \cite{IzVi03}.

Let $u_{c}=(\eta _{c},\xi _{c})\in X$ be a non-constant solution of $\nabla
_{u}\hat{H}(u_{c};0,c)=0$. The $G$-orbit of $u_{c}\in X$ is the manifold%
\begin{equation*}
G(u_{c})=\{u_{c}(x+\varphi )\in X:\varphi \in G\}.
\end{equation*}%
If $\partial _{x}u_{c}\in X$, then the orbit $G(u_{c})$ is a differential
manifold and the tangent space to the orbit at $u_{c}$ is 
\begin{equation*}
T_{u_{c}}G(u_{c})=\{r\partial _{x}u_{c}\in X:r\in \mathbb{R}\}.
\end{equation*}%
Since $H(u;0,c)$ is $G$-invariant, then $\nabla _{u}\hat{H}(\theta \cdot
u_{c};0,c)=0$ for any $\theta \in G$. That is, the orbit $G(u_{c})$ is a set
of critical points of $H$. Moreover, we have 
\begin{equation*}
0=\frac{d}{d\theta }\nabla _{u}\hat{H}(\theta \cdot u_{c};0,c)|_{\theta
=0}=D_{u}^{2}\hat{H}(u_{c};0,c)\partial _{x}u_{c}\text{.}
\end{equation*}%
Therefore, the function $\partial _{x}u_{c}$ is always in the kernel of the
Hessian $D^{2}\hat{H}(u_{c};0,c)$.

\begin{definition}
\label{nondeg}We say that the $G$-orbit of a non-constant solution $u_{c}\in
X$ of $\nabla _{u}\hat{H}(u;0,c)=0$ is \textbf{non-degenerate} if the
Hessian $D^{2}H(u_{c};0,c)$ is a Fredholm operator and its kernel is
generated by $\partial _{x}u_{c}\in X$, 
\begin{equation*}
\ker D^{2}H(u_{c};0,c)=T_{u_{c}}G(u_{c})\text{.}
\end{equation*}
\end{definition}

The kernel of a self-adjoint operator is perpendicular to its range. Define
the $L^{2}$-orthogonal complement to $\partial _{x}u_{c}\in L^{2}$ as%
\begin{equation*}
W=\{w\in L^{2}(S^{1};\mathbb{R)}:\left\langle w,\partial
_{x}u_{c}\right\rangle _{L^{2}}=0\}.
\end{equation*}%
If the $G$-orbit of a non-constant solution $u_{c}\in X$ is non-degenerate,
then the orthogonal complement to the kernel of $D^{2}H(u_{c};0,c)$ is $%
W\cap X$ and the range is $W\cap Y$. Therefore, the Fredholm property
implies that $D_{w}^{2}H(u_{c};0,c):W\cap X\rightarrow W\cap Y$ has a
bounded inverse,%
\begin{equation*}
\left\Vert D_{w}^{2}H(u_{c};0,c)^{-1}\right\Vert _{W\cap Y\rightarrow W\cap
X}\leq M.
\end{equation*}

We define the neighborhood $W_{0}$ of $0\in W\cap X$ as%
\begin{equation*}
W_{0}=\left\{ w\in W\cap X:\left\vert w\right\vert _{X}<\delta \right\} .
\end{equation*}%
Let $\mathcal{U}$ be a $\delta $-neighborhood of the orbit $G(u_{c})\subset
X $. If $u_{c}$ is a function with minimal period $2\pi $, then one can
consider the coordinates $\upsilon :G\times W_{0}\rightarrow \mathcal{U}$
given explicitly in Fourier components by 
\begin{equation}
\upsilon (\theta ,w)=\theta \cdot \left( u_{c}+w\right) =\sum_{k\in \mathbb{Z%
}}e^{i\theta k}\left( \hat{u}_{c,k}+\hat{w}_{k}\right) e^{ikx}.  \label{PSC}
\end{equation}%
This map is $G$-equivariant with the natural action of $G$ on $G\times W_{0}$
given by 
\begin{equation*}
\varphi \cdot (\theta ,w)=(\varphi +\theta ,w).
\end{equation*}%
In the case that $u_{c}$ is a function with minimal period $2\pi /p$, this
map is not bijective, but instead it is a covering map with fibers 
\begin{equation*}
\left\{ (\zeta +\theta ,w(t-\zeta )):\zeta \in \mathbb{Z}_{p}\right\} .
\end{equation*}

\begin{theorem}
\label{Thm3}Let $s>1/2$. If the $G$-orbit of a non-constant solution $u_{c}$
is non-degenerate, then there is an $\varepsilon >0$ such that for $%
\left\vert b\right\vert _{s+1}<\varepsilon $, the equation $\nabla
H(u;b,c)=0 $ has at least two solutions given by 
\begin{equation*}
u_{b,j}(x)=u_{c}(x+\theta _{j})+\mathcal{O}(\left\vert b\right\vert _{s+1}),
\end{equation*}%
where $\theta _{j}\in G=\mathbb{R}/2\pi \mathbb{Z}$ represents a phase shift
depending on $b$ and $\mathcal{O}(\left\vert b\right\vert _{s+1})\in $ $X$
is of order $\left\vert b\right\vert _{s+1}$. The two solutions are
different (not related by a phase shift) if $b$ is not constant.
\end{theorem}

\begin{proof}
The map $\upsilon $ provides new coordinates $(\theta ,w)$ of $\mathcal{U}$
for $(\theta ,w)\in G\times W_{0}$. The Hamiltonian defined in the
coordinates $(\theta ,w)$ is given by 
\begin{equation*}
\mathcal{H}_{b}(\theta ,w):=H(\upsilon (\theta ,w);b,c):G\times
W_{0}\rightarrow \mathbb{R}.
\end{equation*}%
Notice that $\upsilon $ is a covering map when $u_{c}$ has minimal period $%
2\pi /p$, in this case $\mathcal{H}_{b}(\theta ,w)$ is $\mathbb{Z}_{p}$%
-invariant with respect to the action $\zeta \cdot (\theta ,w)=(\zeta
+\theta ,w(t-\zeta ))$ of $\zeta \in \mathbb{Z}_{p}$.

The gradient map $\nabla H\ $is well defined if $\left\vert b\right\vert
_{s+1}\leq \varepsilon $ with $s>1/2$ and $\eta >-h+2\gamma \varepsilon $.
Notice that $\eta _{c}>-h+2\gamma \varepsilon $ if $\varepsilon $ is small
enough because $\eta $ is a continuous solution with $\eta _{c}>-h$. We
denote by $\nabla _{w}\mathcal{H}_{b}:G\times W_{0}\rightarrow W\cap Y$ to
the gradient taken with respect to the variables $w\in W_{0}$. By $G$%
-invariance of the Hamiltonian $\mathcal{H}_{0}(\theta ,w)$ for $b=0$, we
have that $\nabla _{w}\mathcal{H}_{0}(\varphi ,0)=0$ for $\varphi \in G$.
Furthermore, since $(0,0)\in G\times W_{0}$ corresponds to the point $%
u_{c}\in G(u_{c})$, the Hessian $D_{w}^{2}\mathcal{H}_{0}(0,0):W\cap
X\rightarrow W\cap Y$ is invertible with the bound $M$. The fact that $G$
acts by isometries and $\mathcal{H}_{0}$ is $G$-invariant implies that the
Hessian $D_{w}^{2}\mathcal{H}_{0}(\varphi ,0):W\cap X\rightarrow W\cap Y$ is
invertible for all $\varphi \in G$, 
\begin{equation*}
\left\Vert D_{w}^{2}\mathcal{H}_{0}(\varphi ,0)^{-1}\right\Vert _{W\cap
Y\rightarrow W\cap X}\leq M,\qquad \forall \varphi \in G\text{.}
\end{equation*}

We can make a Lyapunov-Schmidt reduction to express the normal variables $%
w\in W_{0}$ in terms of the\ variables along the orbit $\varphi \in G$.
Specifically, since $D_{w}^{2}\mathcal{H}_{0}(\varphi ,0)$ is invertible
with bound $M$, the Implicit Function Theorem implies that there are open
neighborhoods $\varphi \in U_{\varphi }\subset G$ and 
\begin{equation*}
0\in B_{\varepsilon _{\varphi }}:=\{b\in H_{0}^{s+1}:\left\vert b\right\vert
_{s+1}<\varepsilon _{\varphi }\},
\end{equation*}%
such that there is a unique map $w_{\varphi }:U_{\varphi }\times
B_{\varepsilon _{\varphi }}\rightarrow W_{0}$ satisfying $\nabla _{w}%
\mathcal{H}_{b}(\theta ,w_{\varphi }(\theta ;b))=0$. By the compactness of $%
G $ we have that $\varepsilon :=\min_{\varphi \in G}\varepsilon _{\varphi
}>0 $. By uniqueness of $w_{\varphi }(\theta ;b)$, we can glue the functions 
$w_{\varphi }(\theta ;b)$ together to define the map $w(\theta
;b)=w_{\varphi }(\theta ;b):G\times B_{\varepsilon }\rightarrow W$. Thus $%
w:G\times B_{\varepsilon }\rightarrow W$ is the unique map that solves 
\begin{equation*}
\nabla _{w}\mathcal{H}_{b}(\theta ,w(\theta ;b))=0,\qquad \left\vert
b\right\vert _{s+1}<\varepsilon .
\end{equation*}

We define the reduced Hamiltonian by 
\begin{equation*}
h_{b}(\theta )=\mathcal{H}_{b}(\theta ,w(\theta ;b)):G\rightarrow \mathbb{R}.
\end{equation*}%
Since 
\begin{equation*}
\nabla _{\theta }h_{b}(\theta )=\partial _{w}\mathcal{H}_{b}(\theta
,w(\theta ;b))\partial _{\theta }w(\theta ;b)+\partial _{\theta }\mathcal{H}%
_{b}(\theta ,w(\theta ;b))=\partial _{\theta }\mathcal{H}_{b}(\theta
,w(\theta ;b)),
\end{equation*}%
the critical points of the reduced Hamiltonian $h_{b}$ correspond to the
critical points of $\mathcal{H}_{b}$ for $\left\vert b\right\vert
_{s+1}<\varepsilon $. Notice that $h_{b}(\theta )=\mathcal{H}_{b}(\theta
,0)=H(u_{c}(x+\theta );b,c)$ is constant in $\theta $ only if $b$ is
constant. Since $h_{b}:G\rightarrow \mathbb{R}$ is $\mathbb{Z}_{p}$%
-invariant, we obtain critical points by reducing the Hamiltonian $h_{b}$ to
the quotient space $G/\mathbb{Z}_{p}\simeq S^{1}$. Therefore, the reduced
function $h_{b}$ in $G/\mathbb{Z}_{p}$ has at least two different critical
points when $b$ is not constant: one maximum $\theta _{1}$ and one minimum $%
\theta _{2}$. Furthermore, since the map $w(\theta ;b)$ satisfies $w(\theta
;0)=0$ for every $\theta \in G$, then $w(\theta _{j};b)=\mathcal{O}%
(\left\vert b\right\vert _{s+1})$ and 
\begin{equation*}
u_{b,j}(x):=\upsilon (\theta _{j},w(\theta _{j};b))=u_{c}(x+\theta _{j})+%
\mathcal{O}(\left\vert b\right\vert _{s+1}).
\end{equation*}
\end{proof}

\begin{remark}
In the previous theorem, if the solution $u_{c}$ is $2\pi /p$-periodic and
the bottom $b$ is $2\pi /q$-periodic, then the reduced Hamiltonian $%
h_{b}(\theta )$ is $\mathbb{Z}_{q}\times \mathbb{Z}_{p}$-invariant and the
solutions of $\nabla H(u;b,c)=0$ appear in sets of $\mathbb{Z}_{q}$-orbits
that consist of the $2\pi /q$-phase shifts of a steady wave $u_{b,j}$.
\end{remark}

\begin{remark}
The Palais Theorem \cite{Pa} proves, in the case of a compact group $G$
acting in a finite-dimensional Hilbert spaces $X$, that there is always a $G$%
-equivariant map $\upsilon :G\times _{G_{u_{c}}}W_{0}\rightarrow \mathcal{U}$%
, which is called the Palais-slice coordinate map, where $\mathcal{S}%
=\{\upsilon (0,w)\in \mathcal{U}:w\in W_{0}\}$ is called a \emph{slice} and $%
\mathcal{U}$ a \emph{tube} of the orbit $G(u_{c})$. The Palais Theorem is
not applicable to infinite-dimensional Hilbert spaces, but in our case this
map exists because action of $G$ is lineal. Indeed, if $u_{c}$ is a function
with minimal period $2\pi $, then the isotropy $G_{u_{c}}$ is trivial and
the Palais-slice coordinate $\upsilon :G\times W_{0}\rightarrow \mathcal{U}$
is given explicitly in Fourier components by the map (\ref{PSC}). Other
applications of Palais-slice coordinate to bifurcation in Hamiltonian
systems can be found in \cite{FoMo} and references therein.
\end{remark}

\begin{remark}
One can consider also local coordinates in neighborhoods $\mathcal{U}%
_{\theta }$ of $\theta \cdot u_{c}$ that will cover the $\delta $%
-neighborhood $\mathcal{U}$ of the compact orbit $G(u_{c})$. For instance,
one of those local coordinate maps $\upsilon _{\theta }$ is given by%
\begin{equation*}
\upsilon _{\theta }(x)=\theta \cdot u_{c}+x:X_{0}\rightarrow \mathcal{U}%
_{\theta }
\end{equation*}%
with $X_{0}=\left\{ x\in X:\left\vert x\right\vert _{X}<\delta \right\} $.
The application of these local coordinate maps to problems in PDEs can be
found in \cite{Am} and \cite{Kap}, and references therein. The disadvantage
of working with these local coordinate maps with respect to Palais-slice
coordinates is that one loses track of the natural $G$-equivariance of the
problem. Actually, the paper \cite{Kap} proves in Lemma 4.3.3 using the
Implicit Function Theorem that there is a unique map $\left( \tau ,w\right) :%
\mathcal{U}\rightarrow \mathbb{R}\times W$ such that for any $\upsilon
_{\theta }(x)\in \mathcal{U}$ one has $\upsilon _{\theta }(x)=\tau \cdot
u_{c}+w$ for some $\tau \in \mathbb{R}$ and $w\in W=\left\langle \partial
_{x}u_{c}\right\rangle ^{\perp }$, which is the representation that we give
explicitly in (\ref{PSC}).
\end{remark}

\vskip0.25cm \textbf{Acknowledgements.} W.C. was partially supported by the
Canada Research Chairs Program and NSERC through grant number 238452--16.
C.G.A was partially supported by a UNAM-PAPIIT project IN115019. We
acknowledge the assistance of Ramiro Chavez Tovar with the preparation of
the figures. C.G.A is indebted to M. Fontaine, P. Panayotaros, C.R. Barrera
and R.M. Vargas-Maga\~{n}a for discussions related to this project.

\end{document}